\documentclass{article}

\usepackage[a4paper,top=4.0cm,bottom=2.0cm,left=2.0cm,right=2.0cm]{geometry}

\usepackage{natbib}
\bibpunct{(}{)}{;}{d}{,}{,}

\usepackage{amsthm,amsmath,amsfonts,mathrsfs}
\usepackage{epsfig,graphicx}

\DeclareMathOperator{\erf}{erf}
\newtheorem{lemma}{Lemma}
\newtheorem{theorem}{Theorem}

\title{A non-parametric ensemble transform method for Bayesian inference\thanks{Universit\"at Potsdam, 
Institut f\"ur Mathematik, Am Neuen Palais 10, D-14469 Potsdam, Germany}}
\author{Sebastian Reich}

\begin{document}

\maketitle

\begin{abstract} Many applications, such as intermittent data
  assimilation, lead to a  recursive application of Bayesian inference
  within a Monte Carlo context. Popular data assimilation algorithms include
  sequential Monte Carlo methods and ensemble Kalman filters (EnKFs). These
  methods differ in the way Bayesian inference is
  implemented. Sequential Monte Carlo methods rely on importance
  sampling combined with a resampling step while EnKFs
  utilize a linear transformation of Monte Carlo
  samples based on the classic Kalman filter. While EnKFs have proven to be quite robust even for small
  ensemble sizes, they are not consistent since their derivation
  relies on a linear regression {\it ansatz}.  In this paper, we propose another transform
  method, which  does not rely on any a prior assumptions on the
  underlying prior and posterior distributions. The new method is
  based on solving an optimal transportation problem for
  discrete random variables. 
\end{abstract}


\noindent {\bf AMS(MOS) subject classifications.} 65C05, 62M20, 93E11, 62F15, 86A22


\section{Introduction}

This paper is concerned with a particular implementation of Monte
Carlo methods for Bayesian inference and its application to filtering
and intermittent data assimilation \citep{sr:jazwinski}. More specifically, we consider
the problem of estimating posterior expectation values
under the assumption that a finite-size ensemble $\{x_i^f\}_{i=1}^M$
from the (generally unknown) prior distribution $\pi_{X^f}$ is available. A standard approach for obtaining such
estimators relies on the idea of importance sampling based on the
likelihood $\pi_Y(y_0 | x_i^f)$ of the samples $x_i^f$ with regard to a given observation
$y_0$ \citep{sr:Doucet,sr:arul02,sr:crisan}. If applied recursively,
it is necessary to combine importance sampling with a resampling step
such as monomial or systematic resampling \citep{sr:arul02,sr:kuensch05}. More recently the
ensemble Kalman filter (EnKF) has been introduced \citep{sr:evensen}, which transforms the
prior ensemble $\{x_i^f\}_{i=1}^M$ into an uniformly weighted posterior ensemble
$\{x_i^a\}_{i=1}^M$ using the classic Kalman update step of linear
filtering \citep{sr:jazwinski}. The EnKF leads, however, to a
biased estimator even in the limit $M\to \infty$
\citep{sr:lei11}. In this paper, we propose a non-random ensemble
transform method which is based on finite-dimensional optimal
transportation in form of linear programming
\citep{sr:strang,sr:cotterreich}. We provide numerical and theoretical
evidence that the new ensemble transform method
leads to consistent posterior estimators. The new transform method can be
applied to intermittent data assimilation leading to a novel
implementation of particle filters. We demonstrate this possibility
for the chaotic Lorenz-63 model \citep{sr:lorenz63}.

An outline of the paper is as follows. In Section \ref{sec2}, importance
sampling Monte Carlo is summarized in the context of Bayesian
inference. Subsequently importance sampling is put into the context
of linear programming in Section \ref{sec3}. This leads to a novel resampling
method which maximizes the correlation between the prior and posterior
ensemble members. We propose a further modification which turns the
resampling step into a deterministic and linear
transformation. Convergence of the proposed transformation step is
demonstrated numerically by means of two examples. A theoretical
convergence result is formulated based on results by
\cite{sr:mccann95}. Finally, the application to
sequential Monte Carlo methods is discussed in Section \ref{sec4} and
a novel ensemble transform filter is proposed. Numerical results are
presented for the Lorenz-63 model.

\section{Bayesian inference and importance sampling} \label{sec2}

We summarize the importance sampling approach to Bayesian
inference. Given a prior (or in the context of dynamic models,
forecasted) random variable $X^f:\Omega \to \mathbb{R}^N$, we denote
its probability density function (PDF)
by $\pi_{X^f}(x)$, $x\in \mathbb{R}^N$, and consider the assimilation
of an observed $y_0 \in \mathbb{R}^K$ with
likelihood function $\pi_Y(y|x)$. According to Bayes' theorem the analyzed, posterior PDF is given by
\begin{equation} \label{Bayes}
\pi_{X^a}(x|y_0) = \frac{\pi_Y(y_0|x)\pi_{X^f}(x)}{\int_{\mathbb{R}^N}
  \pi_Y(y_0|x)\pi_{X^f}(x)dx}.
\end{equation}

Typically, the forecasted random variable $X^f$ and its PDF are not
available explicitly. Instead one  assume that an ensemble of
forecasts $x_i^f \in \mathbb{R}^N$, $i=1,\ldots,M$, is given, which
mathematically are considered as realizations $X_i^f(\omega)$, $\omega
\in \Omega$, of $M$ independent (or dependent) random variables $X_i^f:\Omega
\to \mathbb{R}^N$ with law $\pi_{X^f}$. Then the expectation value $\bar g^f = \mathbb{E}_{X^f}[g]$ of
a function $g:\mathbb{R}^N \to \mathbb{R}$ with respect to the prior
PDF $\pi_{X^f}(x)$ can be estimated according to
\[
\bar G_M^f = \frac{1}{M} \sum_{i=1}^M g(X_i^f)
\]
with realization
\[
\bar g_M^f = \bar G_M^f(\omega) = \frac{1}{M}\sum_{i=1}^M g(X_i^f(\omega)) = \frac{1}{M}\sum_{i=1}^M g(x_i^f)
\]
for the ensemble $\{x_i^f = X_i^f(\omega)\}_{i=1}^M$.
The estimator is unbiased for any $M>0$ and its variance vanishes as
$M\to \infty$ provided the variance of $g$ is finite.

Following the idea of importance sampling \citep{sr:Liu}, one obtains
the following estimator with respect to the posterior PDF
$\pi_{X^a}(x|y_0)$ using the forecast ensemble:
\[
\bar g_M^a = \sum_{i=1}^M w_i g(x_i^f),
\]
with weights
\begin{equation} \label{weights}
w_i =
\frac{\pi_Y(y_0|x_i^f)}{\sum_{i=1}^M \pi_Y(y_0|x_i^f)}.
\end{equation}
The estimator is no longer unbiased for finite $M$ but remains consistent.
Here an estimator is called consistent if the root mean square error
between the estimator $\bar g_M^a$ and the
exact expectation value $\bar g^a = \mathbb{E}_{X^a}[g]$ vanishes as $M\to \infty$.


\section{An ensemble transform method based on linear programming} \label{sec3}

Alternatively to importance sampling, we may attempt to transform the samples $x_i^f =
X_i^f(\omega)$ with $X_i^f \sim \pi_{X^f}$ into samples $\hat x_i^a$ which follow
the posterior distribution $\pi_{X^a}(x|y_0)$. Then we are back to an estimator
\[
\bar g_M^a = \frac{1}{M} \sum_{i=1}^M g(\hat x_i^a)
\]
with equal weights for posterior expectation values. For univariate
random variables $X^f$ and $X^a$ with PDFs $\pi_{X^f}$ and
$\pi_{X^a}$, respectively, the transformation is characterized by
\begin{equation} \label{cumm}
F_{X^a}(\hat x^a_i) = F_{X^f}(x_i^f),
\end{equation}
where $F_{X^f}$ and $F_{X^a}$ denote the cumulative distribution functions of $X^f$
and $X^a$, respectively, \emph{e.g.}
\[
F_{X^f}(x) = \int_{-\infty}^x \pi_{X^f}(x')dx'.
\]
Eq.~(\ref{cumm}) requires knowledge of the associated PDFs and its
extension to multivariate random variables is non-trivial. In this
section, we propose an alternative approach that does not require
explicit knowledge of the underlying PDFs and that easily generalizes
to multivariate random variables. To obtain
the desired transformation we utilize the idea of optimal
transportation \citep{sr:Villani,sr:Villani2} with respect to an appropriate distance $d(x,x')$ in
$\mathbb{R}^N$. More precisely, we first seek a coupling between two
discrete random variables ${\rm X}^f: \Omega' \to {\cal X}$ and ${\rm X}^a:\Omega' \to
{\cal X}$ with realizations in ${\cal X} = \{x_1^f,\ldots,x_M^f\}$ and
probability vector $p^f = (1/M,\ldots,1/M)^T$ for ${\rm X}^f$ and $p^a
= (w_1,\ldots,w_M)^T$ for ${\rm X}^a$, 
respectively. A coupling between ${\rm X}^f $ and ${\rm X}^a$ is an
$M\times M$ matrix ${\bf T}$ with non-negative
entries $t_{ij} = ({\bf T})_{ij}\ge 0$ such that
\begin{equation} \label{marginals}
\sum_{i=1}^M t_{ij} = 1/M \qquad \sum_{j=1}^M t_{ij} = w_i .
\end{equation}
We now seek the coupling ${\bf T}^\ast$ that minimizes the expected distance
\begin{equation} \label{OT}
\mathbb{E}_{{\rm X}^f{\rm X}^a}[d(x^f,x^a)] = \sum_{i,j=1}^M t_{ij} d(x_i^f,x_j^f) .
\end{equation}
The desired coupling ${\bf T}^\ast$ is characterized by a linear programming
problem \citep{sr:strang}. Since (\ref{marginals})
leads to $2M-1$ independent constraints the matrix ${\bf T}^\ast$ contains
at most $2M-1$ non-zero entries.  

In this paper, we use the squared Euclidean distance, \emph{i.e.}
\begin{equation} \label{dist}
d(x_i^f,x_j^f) = \|x_i^f-x_j^f\|^2.
\end{equation}
We recall that minimizing the expected distance with respect to
the squared Euclidean distance is then equivalent to maximizing
$\mathbb{E}_{{\rm X}^f{\rm X}^a}[(x^f)^T x^a]$ since
\[
\mathbb{E}_{{\rm X}^f{\rm X}^a}[\|x^f-x^a\|^2] = \mathbb{E}_{{\rm X}^f}[\|x^f\|^2]
+ \mathbb{E}_{{\rm X}^a}[\|x^a\|^2] - 2\mathbb{E}_{{\rm X}^f{\rm
    X}^a}[\langle x^f,x^a\rangle] .
\]
with $\langle x^f,x^a \rangle = (x^f)^T x^a$. Furthermore, the optimal coupling ${\bf T}^\ast$ satisfies 
cyclical monotonicity \citep{sr:Villani2}, which is defined as follows. Let $S$ denote the set of all $(x_i^f,x_j^f) \in
{\cal X} \times {\cal X}$ such that $t_{ij}^\ast>0$, then 
\begin{equation} \label{cyclical}
<x_1^a,x_2^f-x_1^f\rangle + \langle x_2^a,x_3^f-x_2^f\rangle + \cdots
+ \langle x_k^a,x_1^f-x_k^f \rangle \le 0
\end{equation}
for any set of pairs $(x_i^f,x_i^a) \in S$, $i=1,\ldots,k$. Any set $S\subset \mathbb{R}^N \times \mathbb{R}^N$ with
this property is called cyclically monotone \citep{sr:Villani2}. 

We next introduce the Markov chain ${\bf P} \in \mathbb{R}^{M\times M}$ on ${\cal X}$ via
\[
{\bf P}= M\,{\bf T}^\ast
\]
with the property that
\[
p^a = {\bf P} p^f.
\]
Given realizations $x_j^f$, $j=1,\ldots,M$, from the prior PDF, a Monte Carlo resampling step proceeds now as follows: Solve (\ref{OT}) for an optimal coupling
matrix ${\bf T}^\ast$ and define discrete random variables
\begin{equation} \label{resample}
{\rm X}^a_j \sim \left( \begin{array}{c} p_{1j}\\ \vdots \\
    p_{Mj} \end{array} \right)
\end{equation}
for $j=1,\ldots,M$. Here $p_{ij}$ denotes the $(i,j)$th entry of ${\bf
  P}$. Note that the random variables ${\rm X}_j^a$, $j=1,\ldots,M$, are neither independent nor
identically distributed.  A new ensemble of size $M$ is finally obtained by collecting a single realization from each
random variable ${\rm X}_j^a$, \emph{i.e.}
\[
x_j^a := {\rm X}_j^a(\omega)
\]
for $j=1,\ldots,M$. This ensemble of equally
weighted samples allows for the approximation
of expectation values with respect to the posterior distribution
$\pi_{X^a}(x|y_0)$. 

The outlined procedure leads to a particular
instance of resampling with replacement \citep{sr:arul02,sr:kuensch05}. The main difference to
techniques such as monomial or systematic resampling is that the resampling is
chosen such that the expected distance (\ref{OT}) between the prior and posterior
samples is minimized. 

We now propose a further modification which replaces the random resampling step 
and generally avoids obtaining multiple copies in the analyzed ensemble $\{x_i^a\}_{i=1}^M$. The 
modification is based on the observation that
\begin{equation} \label{meanxa}
\bar x_j^a = \mathbb{E}_{{\rm X}_j^a}[x] = \sum_{i=1}^M p_{ij} x_i^f .
\end{equation}
We use this result to propose the deterministic transformation 
\begin{equation} \label{OTtrans}
x_j^a := \bar x_j^a = \sum_{i=1}^M p_{ij} x_i^f
\end{equation}
$j=1,\ldots,M$. The idea is that
\[
\bar g_M^a = \frac{1}{M} \sum_{j=1}^M g(\bar x_j^a)
\]
still provides a consistent estimator for $\mathbb{E}_{X^a}[g]$ as
$M\to \infty$. For the special case $g(x)=x$ it is easy
to verify that indeed
\[
\bar x_M^a= \frac{1}{M} \sum_{j=1}^M x_j^a = \frac{1}{M} \sum_{j=1}^M
\sum_{i=1}^M p_{ij} x_i^f = \sum_{i,j} t_{ij}^\ast x_i^f = \sum_{i=1}^M w_i x_i^f.
\]

Before investigating the theoretical properties of the proposed
transformation (\ref{OTtrans}) we consider two examples which 
indicate that (\ref{OTtrans}) indeed leads to a consistent approximation
to (\ref{cumm}) in the limit $M\to \infty$.

\begin{figure}
\begin{center}
\includegraphics[width=0.5\textwidth]{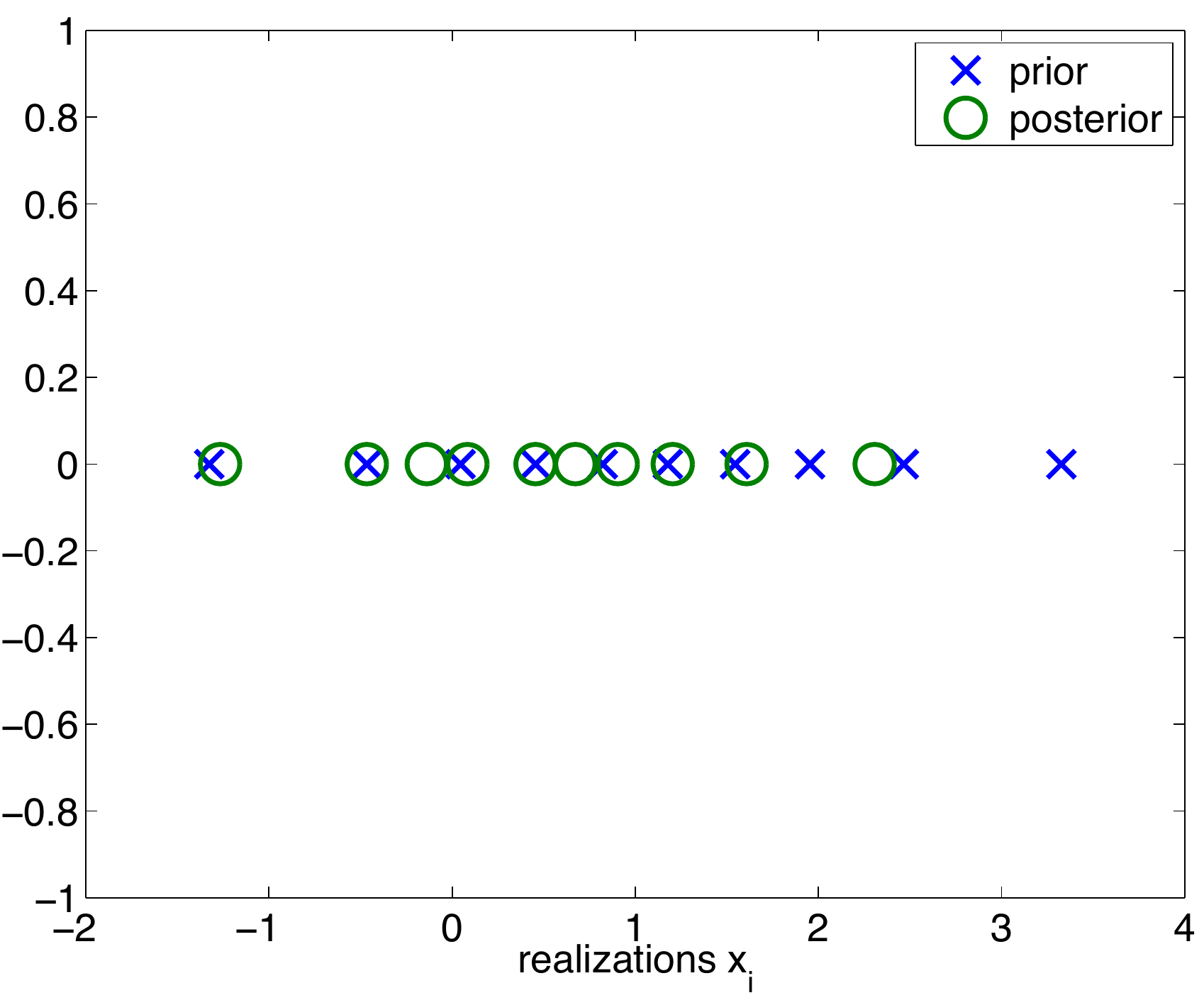}
\end{center}
\caption{Prior $x_i^f$ and posterior $x_i^a$ realizations from the
  transform method for $M=10$.}
\label{figure1}
\end{figure}

\begin{figure}
\begin{center}
\includegraphics[width=0.5\textwidth]{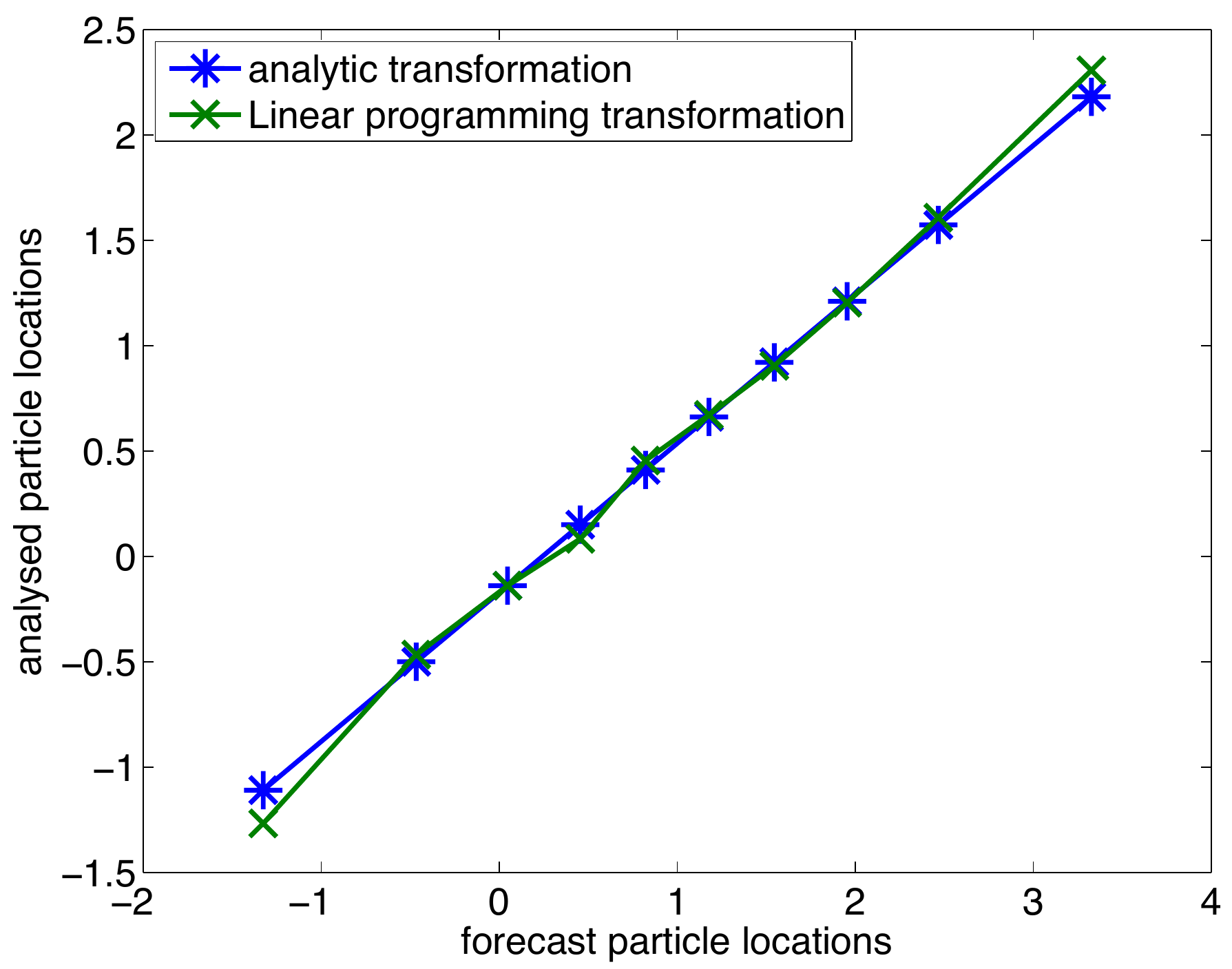}
\end{center}
\caption{Exact and numerical ensemble transform map for $M=10$. The
  Gaussian case leads to the exact transformation being linear. The
  numerical approximation deviates from linearity mostly in its both tails.}
\label{figure1b}
\end{figure}

\begin{figure}
\begin{center}
\includegraphics[width=0.5\textwidth]{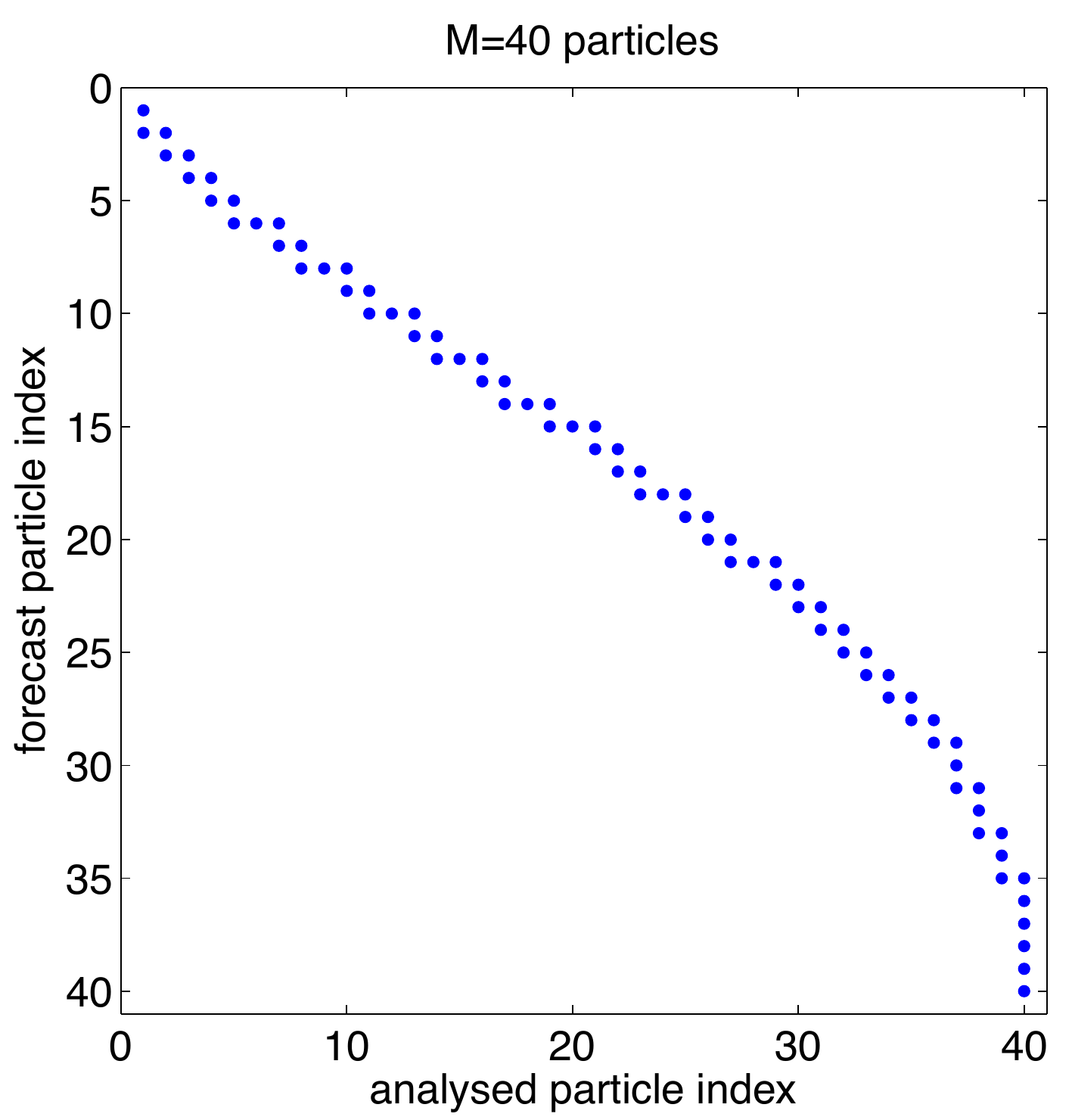}
\end{center}
\caption{Non-zero entries in the matrix ${\bf P}$ for $M=40$, {\it i.e.}~the support of the 
coupling. There are a total of $2M-1 = 79$ non-zero entries. The banded structure
  reveals the spatial locality and the cyclical monotonicity \cite{sr:Villani,sr:Villani2} 
  of the resampling step.}
\label{figure2}
\end{figure}

{\sc Example.} {\it We take the univariate Gaussian with mean $\bar x =1$ and variance $\sigma^2 =
2$ as prior random variable $X^f$. Realizations of $X^f$ are generated using 
\[
x_i^f = \sqrt{2}  \erf^{-1}(2u_i-1), \quad u_i = \frac{1}{2M} +
\frac{i-1}{M}
\]
for $i=1,\ldots,M$. The likelihood function is
\[
\pi_Y(y|x) = \frac{1}{\sqrt{4\pi}} \exp\left( \frac{-(y-x)^2}{4} \right)
\]
with assumed observed value $y_0 = 0.1$. Bayes' formula yields a posterior
distribution which is Gaussian with mean $\bar x = 0.55$ and variance
$\sigma^2 = 1$. The prior and posterior realizations from the
transform method are shown for $M=10$ in Figure \ref{figure1}. We also
display the analytic transform, which is a straight line in case of
Gaussian distributions, and the approximate transform using
linear programming in Figure \ref{figure1b}. The structure of
non-zero entries of the Markov chain matrix ${\bf P}$ for $M=40$ is displayed in
Figure \ref{figure2}, which shows a banded structure of local interactions. The staircase-like arrangement is due to cyclical
monotonicity of the support of ${\bf T}^\ast$. More generally, one obtains the posterior
estimates for the first four moments displayed in Table \ref{table1},
which indicate convergences as $M\to \infty$. }

\begin{table}
\caption{Estimated posterior first to fourth-order moments from the
  ensemble transform method applied to a Gaussian scalar Bayesian
  inference problem.}
\label{table1}
\begin{center}
\begin{tabular}{ccccc}
\hline\noalign{\smallskip}
& $\bar x$ & $\sigma^2$  & $\mathbb{E}[(X-\bar x)^3]$  & $\mathbb{E}[(X-\bar x)^4]$ \\
\noalign{\smallskip}\hline\noalign{\smallskip} 
$M= 10$ & 0.5361 & 1.0898 & -0.0137 & 2.3205 \\
$M= 40$& 0.5473 & 1.0241 & 0.0058 & 2.7954 \\ 
$M = 100$ & 0.5493 & 1.0098 & -0.0037 & 2.9167\\ 
\noalign{\smallskip}\hline
\end{tabular}
\end{center} 
\end{table}


{\sc Example.} {\it As a further (non-Gaussian) example we
consider a uniform prior on the interval $[0,1]$ and use samples
$x_i^f=u_i$ with the $u_i$'s as defined in the previous example. 
Given the observed value $y_0 = 0.1$, the posterior PDF is
\[
\pi_{X^a}(x|0.1) = \left\{ \begin{array}{ll} \frac{1}{0.9427\ldots}
    e^{-(x-0.1)^2/4} & x\in [0,1]\\
0 & \mbox{else} \end{array} \right.
\]
The resulting posterior mean is $\bar x \approx 0.4836$ and its
variance $\sigma^2 \approx 0.0818$. The third and fourth moments are
$0.0016$ and $0.0122$, respectively. The transform method yields the
posterior estimates for the first four moments displayed in Table
\ref{table2}, which again indicate convergences as $M\to \infty$.}

\begin{table}
\caption{Estimated posterior first to fourth-order moments from the
  ensemble transform method applied to a non-Gaussian scalar Bayesian
  inference problem.}
\label{table2}
\begin{center}
\begin{tabular}{ccccc}
\hline\noalign{\smallskip}
& $\bar x$ & $\sigma^2$  & $\mathbb{E}[(X-\bar x)^3]$  & $\mathbb{E}[(X-\bar x)^4]$ \\
\noalign{\smallskip}\hline\noalign{\smallskip} 
$M= 10$ & 0.4838 & 0.0886 & 0.0014 & 0.0114 \\
$M= 40$& 0.4836 & 0.0838 & 0.0016 & 0.0121 \\ 
$M = 100$ & 0.4836 & 0.0825 & 0.0016 & 0.0122\\ 
\noalign{\smallskip}\hline
\end{tabular}
\end{center} 
\end{table}


We now proceed with a theoretical investigation of the transformation (\ref{OTtrans}).
Our convergence result is based on the following lemma and
general results from \cite{sr:mccann95}.

\begin{lemma}  \label{lemma}
The set $\hat S$ consisting of all pairs $(x_j^f,\bar
  x_j^a)$, $j=1,\ldots,M$, with $\bar x_j^a$ defined by (\ref{meanxa}), is cyclically monoton.
\end{lemma}

\begin{proof} Let $I(j)$ denote the set of indices $i$ for which
  $p_{ij}>0$. Since ${\bf T}^\ast$ is cyclically monoton, 
  (\ref{cyclical}) holds for sequences containing a term of type $\langle x_i^f,x_{j+1}^f- x_j^f \rangle$ with $i \in
  I(j)$. By linearity of $\langle x_i^f,x^f_{j+1}-x_j^f\rangle$ in each of its two
  arguments, (\ref{cyclical}) then also applies to linear combinations
  giving rise to
\begin{align*}
 & \sum_{i=1}^M p_{ij} \left\{  \langle x_1^a,x_2^f-x_1^f\rangle  + \langle x_2^a,x_3^f-x_2^f\rangle + 
\cdots +  \langle
x_i^f,x_{j+1}^f - x_j^f \rangle +
\cdots
+ \langle x_k^a,x_1^f-x_k^f \rangle  \right\} =\\
&  \langle x_1^a,x_2^f-x_1^f\rangle  + \langle x_2^a,x_3^f-x_2^f\rangle +  
\cdots + \langle \bar x_j^a,x_{j+1}^f - x_j^f \rangle + \cdots
+ \langle x_k^a,x_1^f-x_k^f \rangle  \le 0 
\end{align*}
since $\sum_{i=1}^M p_{ij} = 1$. The same procedure can be applied
to all indices $j' \in \{1,\ldots,k\} \setminus \{j\}$ resulting in
\[
 \langle \bar x_1^a,x_2^f-x_1^f\rangle  + \langle \bar x_2^a,x_3^f-x_2^f\rangle +  
\cdots + \langle \bar x_j^a,x_{j+1}^f - x_j^f \rangle + \cdots
+ \langle \bar x_k^a,x_1^f-x_k^f \rangle  \le 0 .
\]
Hence the set $\hat S$ is cyclically monoton.
\end{proof}

\begin{theorem}
Assume that the ensemble ${\cal X}_M^f = \{x_i^f\}_{i=1}^M$ consists of
realization from $M$ independent and identically distributed random variables 
$X_i^f:\Omega \to \mathbb{R}^N$ with PDF $\pi_{X^f}$. Define
the set ${\cal X}_M^a = \{\bar x_j^a\}_{j=1}^M$ with the $\bar
x_j^a$'s given by (\ref{meanxa}).
Then the associated maps $\Psi_M:{\cal X}_M^f \to
{\cal X}_M^a$, defined for fixed $M$ by 
\[
\bar x_j^a = \Psi_M(x_j^f), \qquad j=1,\ldots,M,
\]
converge weakly to a map $\Psi:\mathbb{R}^N\to \mathbb{R}^N$ for
$M\to \infty$. Furthermore, the random variable defined by $X^a =
\Psi(X^f)$ has distribution (\ref{Bayes}) and the expected distance between
$X^a$ and $X^f$ is minimized among all such mappings.
\end{theorem}

\begin{proof} The maps $\Psi_M$ define a sequence of couplings 
between discrete random variables on ${\cal X}_M^f$ and ${\cal X}_M^a$,
  which satisfy cyclical monotonicity according to Lemma \ref{lemma}.
We may now follow the proof of Theorem 6 of \cite{sr:mccann95} to conclude that 
these couplings converge weakly to a coupling, \emph{i.e.}~a
probability measure $\mu_{X^fX^a}$ on $\mathbb{R}^N\times \mathbb{R}^N$ with marginals
$\pi_{X^f}$ and $\pi_{X^a}$, respectively. Furthermore, $\mu_{X^f X^a}$ is also
cyclically monoton and the Main Theorem of \cite{sr:mccann95} can be
applied to guarantee the existence of the map $\Psi$, which itself is
the gradient of a convex potential $\psi$. The coupling $\mu_{X^f X^a}$
solves the Monge-Kantorovitch problem with cost $c(x,y) = \|x-y\|^2$ 
\citep{sr:Villani,sr:Villani2}.
\end{proof}

One may replace the uniform probabilities in $p^f$ by an appropriate
random vector $p^f = (w_1^f,\ldots,w_M^f)^T$, i.e.~$w_i^f\ge 0$ and
$\sum_{i=1}^M w_i^f = 1$. To clarify the notations we write $p^a =
(w_1^a,\ldots,w_M^a)^T$ for the posterior weights according to
Bayes' formula. The linear programming problem (\ref{OT}) is 
adjusted accordingly and one obtains an optimal coupling ${\bf
  T}^\ast$ and an induced Markov chain ${\bf P}$ with entries
\[
p_{ij} = \frac{t_{ij}}{w_j^f}.
\]
Hence the transform method (\ref{OT}) is now replaced by
\begin{equation} \label{transform1}
\bar x_j^a = \sum_{i=1}^M p_{ij} x_i^f
\end{equation} 
and the posterior ensemble mean satisfies
\[
\bar x^a_M = \sum_{j=1}^M w_j^f \bar x_j^a = \sum_{j=1}^M \sum_{i=1}^M
w_j^f \frac{t_{ij}}{w_j^f} x_i^f = \sum_{i=1}^M w_i^a x_i^f 
\]
as desired. More generally, posterior expectation values are given by
\[
\bar g_M^a = \sum_{i=1}^M w_i^f g(x_i^a).
\]


\section{Application to sequential data assimilation} \label{sec4}

We now apply the proposed ensemble transformation (ET) method (\ref{OT}) to sequential state
estimation for ordinary differential equation models
\begin{equation} \label{ode}
\dot{x} = f(x)
\end{equation}
with known PDF $\pi_0$ for the initial conditions $x(0)\in \mathbb{R}^N$ at time
$t=0$. Hence we treat solutions $x(t)$ as realizations of the random variables
$X_t$, $t\ge 0$, determined by the flow of (\ref{ode}) and the initial PDF $\pi_0$.

We assume the availability of observations $y(t_k) \in \mathbb{R}^K$ at discrete
times $t_k = k\Delta t_{\rm obs}$, $k>0$, in intervals of $\Delta t_{\rm obs}>0$.
The observations satisfy the forward model
\[
Y(t_k) = h(x_{\rm ref}(t_k)) + \Xi_k,
\]
where $\Xi_k:\Omega \to \mathbb{R}^K$ represent independent and identically distributed
centered Gaussian random variables with covariance matrix
$R \in \mathbb{R}^{K\times K}$, $h:\mathbb{R}^N \to \mathbb{R}^K$
is the forward map, and $x_{\rm ref}(t) \in \mathbb{R}^N$ denotes the desired
reference solution. The forward model gives rise to the likelihood
\[
\pi_Y(y|x) = \frac{1}{(2\pi)^{K/2}|R|^{1/2}} \exp\left(-\frac{1}{2}
(y-h(x))^T R^{-1} (y-h(x)) \right).
\]

A particle filter starts from an ensemble $\{x_i(0)\}_{i=1}^M$ of $M$ realizations
from the initial PDF $\pi_0$. We evolve this ensemble of realizations
under the model dynamics (\ref{ode}) till the first observation
$y_{\rm obs}(\Delta t_{\rm obs})$ becomes available at which point we apply the proposed
ET method to the forecast ensemble members $x_i^f = x_i(\Delta
t_{\rm obs})$. If one furthermore collects these prior realizations into an $N\times M$ matrix
\[
{\bf X}^f = [x_1^f \cdots x_M^f],
\]
then, for given observation $y_0 = y(\Delta t_{\rm obs})$, the
ET method (\ref{transform1}) leads to the posterior realizations simply given by
\begin{equation} \label{transform2}
{\bf X}^a = {\bf X}^f {\bf P}, \qquad [x_1^a \cdots x_M^a ] = {\bf X}^a,
\end{equation}
where ${\bf P}$ is the Markov chain induces by the associated linear
programming problem. The analysed ensemble members $x_i^a$, $i=1,\ldots,M$, are now being used as new
initial conditions for the model (\ref{ode}) and the process of
alternating between propagation under model dynamics and assimilation of data is repeated
for all $k>1$.

It should be noted that a transformation similar to (\ref{transform2}) arises
from the ensemble square root filter (ESRF)
\citep{sr:evensen}. However, the transform matrix ${\bf P} \in
\mathbb{R}^{M\times M}$ used here
is obtained in a completely different manner and does not relly on the
assumption of the PDFs being Gaussian. We mention the work of
\cite{sr:lei11} for an alternative approach to modify EnKFs in order to make
them consistent with non-Gaussian distributions.
We now provide a numerical example and compare an ESRF implementation
with a particle filter using the new ET method.

\begin{figure}
\begin{center}
\includegraphics[width=0.5\textwidth]{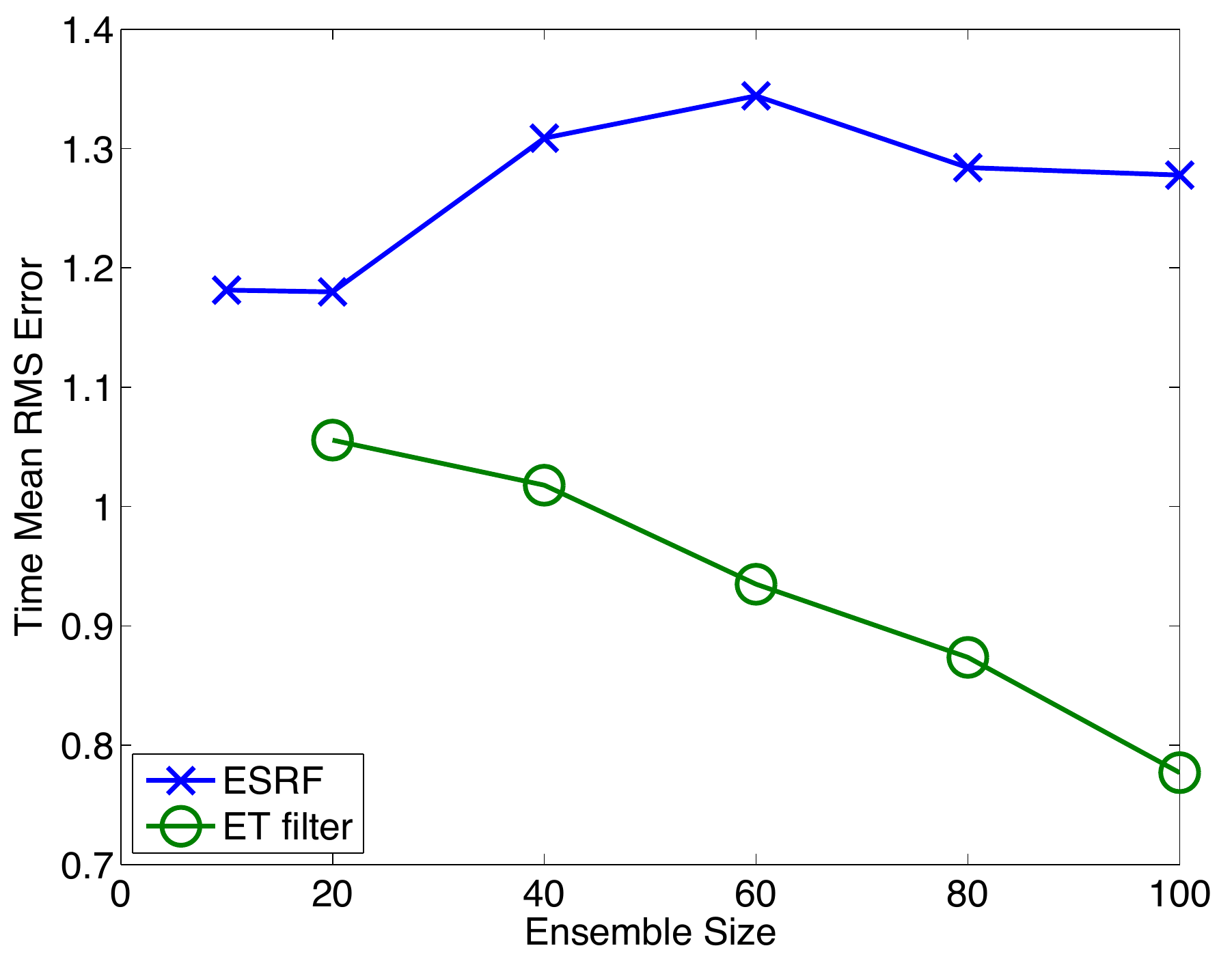}
\end{center}
\caption{Time averaged RMS errors for the Lorenz-63 model in the setting
  of \cite{sr:anderson10} for an ensemble square root
filter (ESRF) and the new ensemble transform (ET) filter for increasing ensemble sizes $M$.}
\label{figure3}
\end{figure}

{\sc Example.} {\it We consider the Lorenz-63 model \cite{sr:lorenz63}
\begin{align*}
\dot{\rm x} &= \sigma ({\rm y}-{\rm x}),\\
\dot{\rm y} &= {\rm x}(\rho-{\rm z}) - {\rm y},\\
\dot{\rm z}& = {\rm xy} - \beta {\rm z}
\end{align*}
in the parameter and data assimilation setting of \cite{sr:anderson10}. In particular,
the state vector is $x = ({\rm x},{\rm y},{\rm z})^T \in \mathbb{R}^3$
and we observe all three variables every $\Delta t_{\rm obs} = 0.12$ time units with a measurement error 
variance $R=8$ in each observed solution component. The equations are
integrated in time by the implicit midpoint rule with step-size
$\Delta t = 0.01$. We implement an ESRF \cite{sr:evensen} and the new 
ET filter for ensemble sizes $M=10,20,40,60,80,100$. The results for
both methods use an optimized form of ensemble 
inflation \citep{sr:evensen}. The ET nevertheless leads to filter
divergence for $M=10$ while the ESRF is stable for all given choices of $M$. 
The time averaged root mean square (RMS) errors over 2000 assimilation
steps can be found in Fig.~\ref{figure3}. It is evident that the new ET filter
leads to much lower RMS errors for all $M\ge 20$. The results also
compare favorable to the ones displayed in \cite{sr:anderson10}
for the rank histogram filter (RHF) \cite{sr:anderson10} and the EnKF
with perturbed observations \citep{sr:evensen}.}


\section{Conclusions}

We have explored the application of linear programming and optimal
transportation to Bayesian inference and particle filters. We have
demonstrated theoretically as well as numerically that the
proposed ET method allows to reproduce posterior expectation values in
the limit $M\to \infty$ and a convergence to the associated continuum
optimal transport problem \citep{sr:Villani,sr:Villani2}. 
The application of continuous optimal transportation to Bayesian inference
has been discussed by \cite{sr:marzouk11},
\cite{sr:reich10,sr:reich11}, \cite{sr:cotterreich}. However, a direct
application of continuous optimal transportation to Bayesian inference
in high-dimensional state spaces $\mathbb{R}^N$ seems currently out of
reach and efficient numerically techniques need to be developed. 
It remains to investigate what modifications are required (such
as localization \citep{sr:evensen}) in
order to implement the proposed ET method even if the ensemble sizes $M$ are much smaller than the
dimension of state space $N$ (or the dimension of the attractor of
(\ref{ode}) in case of intermittent data assimilation). A standard Matlab implementation of the simplex algorithm
was used for solving the linear programming problems in this paper. More efficient algorithms such as the 
auction algorithm \citep{sr:bertsekas89} should be considered in future implementations of the ET method (\ref{OT}). 

\section*{Acknowledgments}
I would like to thank Yann Brenier and Jacques Vanneste for
inspiring discussions on the subject of this paper during an
Oberwolfach workshop in 2010. The paper benefited furthermore from
additional discussions with Wilhelm Stannat and Dan Crisan at another
Oberwolfach workshop in 2012.


\bibliographystyle{plainnat}
\bibliography{survey}

\end{document}